\documentclass[11pt,letterpaper,twoside,reqno,nosumlimits]{amsart}
\usepackage{graphicx}
\usepackage{amssymb,bm}
\usepackage{setspace}
\usepackage{colonequals}
\usepackage{qbordermatrix}

\newcommand{\mat}[1]{\ensuremath{\bm{#1}}}
\renewcommand{\vec}[1]{\ensuremath{\bm{#1}}}
\newcommand{\e}{\ensuremath{\mathrm{e}}}
\newcommand{\E}{\ensuremath{\mathbb{E}}}
\newcommand{\Prob}[1]{\ensuremath{\mathbb{P}\left\{#1\right\}}}

\newcommand{\R}{\ensuremath{\mathbb{R}}}

\newcommand{\norm}[1]{\ensuremath{\big\|#1\big\|}}

\newtheorem{thm}{Theorem}
\newtheorem{prop}{Proposition}
\newtheorem{lemma}{Lemma}
\newtheorem{cor}{Corollary}

\theoremstyle{remark}
\newtheorem{remark}{Remark}
\title[The na\"ive Nystr\"om extension]{The spectral norm error of the na\"ive Nystr\"om extension}

\author{Alex Gittens}
\thanks{Research supported, under the auspice of Joel Tropp, by ONR awards N00014-08-1-0883 and N00014-11-1-0025, AFOSR award FA9550-09-1-0643, and a Sloan Fellowship. In addition to his financial support, the author thanks Joel for fruitful discussions on the interpretation of these error bounds.}

\begin{document}
\begin{abstract}The na\"ive Nystr\"om extension forms a low-rank approximation to a positive-semidefinite matrix by uniformly randomly sampling from its columns. This paper provides the first relative-error bound on the spectral norm error incurred in this process. This bound follows from a natural connection between the Nystr\"om extension and the column subset selection problem. The main tool is a matrix Chernoff bound for sampling without replacement.
\end{abstract}
\maketitle

\section{Introduction}
Nystr\"om extensions are a class of algorithms that quickly form low-rank approximations to positive semidefinite (PSD) matrices by sampling from their columns. We consider the na\"ive Nystr\"om extension, a particular scheme in which the columns are sampled uniformly without replacement. By exploiting a natural connection between Nystr\"om extensions and the column subset selection problem, we find the first relative-error spectral norm guarantees for the na\"ive Nystr\"om extension.

\subsection{Efficacy of the na\"ive Nystr\"om extension}
Perhaps surprisingly, given that one uses no information about the matrix itself to make the column selections, the na\"ive Nystr\"om extension is effective in practice. Because of its data agnosticism and empirical accuracy, the na\"ive Nystr\"om extension is a natural choice for any application where one wishes to avoid the cost of examining (or even constructing) the entire dataset before approximation. 

The na\"ive Nystr\"om extension has proven to be particularly useful in image-processing applications, which typically involve computations with large dense matrices \cite{BCFM04, WDTLG09, BF11}. In spectral image segmentation, for example, one constructs a matrix of pairwise pixel affinities by comparing neighborhoods of each pair of pixels. Several leading eigenvectors of this matrix are then used to segment the image. The affinity matrix of an $N \times N$ image has dimension $N^2 \times N^2,$ so it is challenging to construct and hold the affinity matrix in memory even for images of a moderate size. Similarly, the density and size of the affinity matrix makes it challenging to compute the leading eigenvectors. \cite{BCFM04} proposes using the na\"ive Nystr\"om extension to approximate the eigenvectors of the affinity matrix. Doing so allows one to work with much larger images, because it is only necessary to compute a fraction of the columns of the affinity matrix.  

\subsection{Structure of the Nystr\"om extension} Let $\mat{A}$ be a PSD matrix of size $n.$ Select $\ell \ll n$ columns of $\mat{A}$ to constitute the columns of a matrix $\mat{C}.$ Let $\mat{W}$ be the $\ell \times \ell$ matrix formed by the intersection of the columns in $\mat{C}$ and the corresponding rows in $\mat{A}.$ The matrix $\mat{C} \mat{W}^\dagger \mat{C}^t$ is then a Nystr\"om extension of $\mat{A}$ (see Figure \ref{fig:nystromprocedure}). Here $(\cdot)^\dagger$ denotes Moore-Penrose pseudoinversion. Since $\mat{W}$ is a principal submatrix of $\mat{A},$ it is positive-semidefinite, and hence the Nystr\"om extension is also positive-semidefinite.

\begin{figure}[h]
 \centering
 \includegraphics[scale=.75]{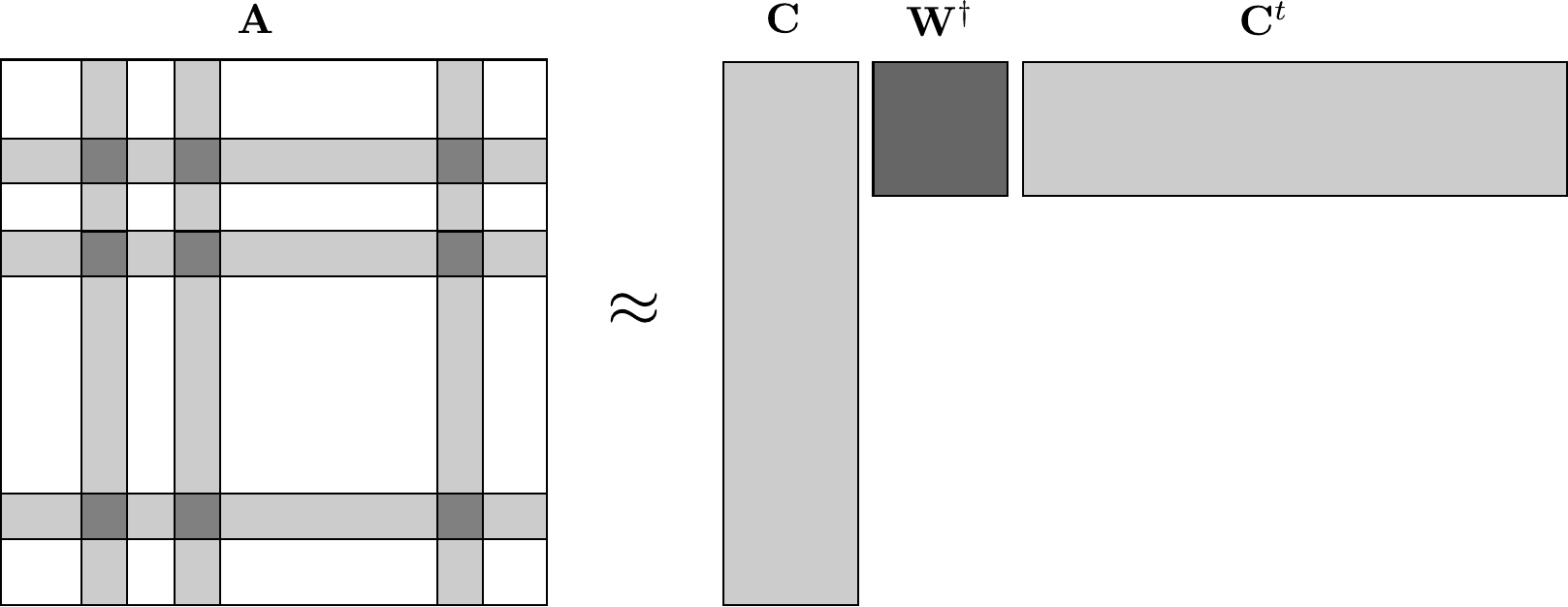}
 \caption{The Nystr\"om extension procedure}
 \label{fig:nystromprocedure}
\end{figure}

The manner in which the columns are sampled and $\mat{W}^\dagger$ is calculated or approximated determines the type of the Nystr\"om extension. Various sampling schemes have been proposed, ranging from the fast and simple na\"ive scheme in which the columns are selected uniformly at random without replacement to more sophisticated and calculation-intensive schemes that involve sampling from a distribution determined by the determinants of principal submatrices of $\mat{A}$ \cite{BW09}. In practice the na\"ive scheme represents a favorable trade-off between speed and accuracy \cite{KMT09}.
 
\subsection{Nystr\"om approximation of invariant subspaces}
In many applications, including the image-processing example taken from \cite{BCFM04}, the Nystr\"om extension is used to obtain approximations to the dominant invariant subspace of a PSD matrix $\mat{A},$ rather than a low-rank approximation \cite{homrighausen11}. 
Through the Davis--Kahan sin$\Theta$ theorem, the spectral norm approximation error provides information on the quality of the approximate invariant subspace obtained via Nystr\"om extensions \cite[Section VII.3]{Bhatia97}. 

To be more precise, let $\tilde{\mat{A}}$ be a Nystr\"om approximation to $\mat{A}$ and assume both $\mat{A}$ and $\tilde{\mat{A}}$ have unique dominant $k$-dimensional invariant subspaces. Let $\mat{U}_k$ and $\tilde{\mat{U}}_k$ have orthogonal columns and span, respectively, the dominant $k$-dimensional invariant subspace of $\mat{A}$ and that of $\tilde{\mat{A}}.$ Recall one natural definition for the distance between the dominant $k$-dimensional invariant subspaces of $\mat{A}$ and $\tilde{\mat{A}},$  
\[
 \text{dist}(\mat{U}_k, \tilde{\mat{U}}_k) = \|\mat{P}_{\mat{U}_k} - \mat{P}_{\tilde{\mat{U}}_k}\|_2.
\]
Denote the $k$th-largest eigenvalue of a matrix $\mat{M}$ by $\lambda_k(\mat{M}),$ so that $\lambda_1(\mat{M}) \geq \lambda_2(\mat{M}) \geq \ldots.$
It follows from the Davis--Kahan sin$\Theta$ theorem that if $\lambda_k(\mat{A}) - \lambda_{k+1}(\tilde{\mat{A}}) > 0,$ then
\begin{equation}
 \text{dist}(\mat{U}_k, \tilde{\mat{U}}_k) \leq \frac{\|\mat{A} - \tilde{\mat{A}}\|_2}{\lambda_k(\mat{A}) - \lambda_{k+1}(\tilde{\mat{A}})}.
\label{eqn:rawdksintheta}
\end{equation}

Assume that we have a relative-error spectral norm bound of the form
\[
 \| \mat{A} - \tilde{\mat{A}} \|_2 \leq \mathrm{C} \lambda_j(\mat{A}),
\]
for some $j \geq k.$ Then equation \eqref{eqn:rawdksintheta} becomes
\[
 \text{dist}(\mat{U}_k, \tilde{\mat{U}}_k) \leq \frac{ \mathrm{C} \lambda_j(\mat{A}) }{\lambda_k(\mat{A}) - \lambda_{k+1}(\mat{A}) - \mathrm{C}\lambda_j(\mat{A})}.
\]
Thus, we conclude that if $\mathrm{C} \lambda_j(\mat{A})$ is sufficiently smaller than the eigengap $\lambda_k(\mat{A}) - \lambda_{k+1}(\mat{A}),$ the Nystr\"om extension yields a quality approximation to the dominant $k$-dimensional invariant subspace of $\mat{A}.$

This paper presents a simple framework for the analysis of Nystr\"om schemes that yields a state-of-the-art spectral norm error bound in the case of the na\"ive Nystr\"om extension scheme. Specifically, it generalizes the coherence-based exact recovery result in \cite{TR10} to also guarantee small relative error in the case of a matrix with a fast-decaying spectrum. This is the first truly relative-error spectral norm bound available for any Nystr\"om extension method. When the eigengap $\lambda_k(\mat{A}) - \lambda_{k+1}(\mat{A})$ is sufficiently large, our result sanctions the use of the na\"ive Nystr\"om extension for the approximation of the dominant $k$-dimensional invariant subspace of $\mat{A}.$

\subsection{Our relative-error spectral norm bound}The efficacy of the na\"ive Nystr\"om extension is of course dependent on the data set to which it is applied. Intuitively, the extension should perform better if the information is spread evenly throughout the columns of the matrix. The \emph{coherence} of the invariant subspaces of $\mat{A}$ provides a quantitative measure of the informativity of the columns. Let $\mathcal{S}$ be a $k$-dimensional subspace of $\R^n$ and $\mat{P}_{\mathcal{S}}$ denote the projection onto $\mathcal{S}.$ Then the coherence of $\mathcal{S}$ is
\[
 \mu_0(\mathcal{S}) = \frac{n}{k} \max\nolimits_i (\mat{P}_{\mathcal{S}})_{ii}.
\]

Corollary \ref{cor:simpleuniformnystromerror} is a condensed version of our main result, Theorem \ref{thm:uniformnystromerror}, and uses the notion of coherence to provide a bound on the error of the na\"ive Nystr\"om extension. 
\begin{cor}
 Let $\mat{A}$ be a real PSD matrix of size $n.$ Given an integer $k \leq n,$ let
$\tau$ denote the coherence of a dominant  $k$-dimensional invariant subspace of $\mat{A}.$
Fix a nonzero failure probability $\delta.$ If $\ell \geq 8 \tau k \log(k/\delta)$ columns of $\mat{A}$ are chosen uniformly at random without replacement, then 
\[
\|\mat{A} - \mat{C} \mat{W}^\dagger \mat{C}^t \|_2 \leq \lambda_{k+1}(\mat{A}) \left(1 + \frac{2n}{\ell} \right) 
\]
with probability exceeding $1-\delta.$
\label{cor:simpleuniformnystromerror}
\end{cor}
For Corollary \ref{cor:simpleuniformnystromerror} to provide a meaningful estimate of $\ell,$ the required number of column samples, the coherence $\tau$ must be small enough that $\ell \ll n.$ This requirement reflects our intuition that approximations formed using a small number of columns will not be accurate if a small number of columns are significantly more influential than the others. In the best-case scenario of $\tau=1,$ the columns are equally informative and we find that the error of a na\"ive Nystr\"om extension formed using just $\mathrm{C} k \log k$ columns is close to that of the optimal rank-$k$ approximation.

\subsection{Relevant literature}

We briefly review the literature on Nystr\"om extensions, focusing on the na\"ive Nystr\"om scheme. In this section $\mat{A}$ is a PSD matrix, $\mat{A}_k$ is a rank-$k$ approximation to $\mat{A}$ that is optimal in the spectral norm, and $\ell$ is the number of columns used to construct a Nystr\"om extension of $\mat{A}.$

Williams and Seeger introduce the Nystr\"om extension in \cite{SW01}, based upon a similar method used in numerical integral equation solvers, as a heuristic method for efficiently approximating the eigendecomposition of kernel matrices. In this seminal work, only an empirical analysis of the approximation error is offered. Drineas and Mahoney provide the first rigorous analysis of a Nystr\"om extension in \cite{DM05}; in the scheme they consider, columns are sampled with probability proportional to the square of the diagonal entries of $\mat{A}$. In addition to probabilistic schemes, many adaptive sampling schemes have been proposed. These attempt to progressively choose the columns to decrease the approximation error. For an introduction to this body of literature, we refer the interested reader to the discussion in \cite{FGK11}.

Kumar, Mohri, and Talwalkar attempt the first analysis of the na\"ive Nystr\"om extension in \cite{KMT09}, resulting in bounds for the Frobenius norm error. Their analysis proceeds by bounding the expectation and variance of the error then applying a concentration of measure argument. A simplified yet representative statement of their bound is that 
\[
 \|\mat{A} - \mat{C}\mat{W}^\dagger\mat{C}^t \|_F \leq \|\mat{A} - \mat{A}_k\|_F + \varepsilon n \cdot \max_i {(\mat{A})_{ii}}
\]
with constant probability when $\ell \geq Ck/\varepsilon^4.$ 

In \cite{KMT09a}, Kumar, Mohri, and Talwalkar establish that if $\text{rank}(\mat{W})= \text{rank}(\mat{A}) = r,$ then $\mat{A} = \mat{C}\mat{W}^\dagger \mat{C}^t.$ Talwalkar and Rostamizadeh prove this implies that, if $\ell \geq \mathrm{C} r \mu \log(r/\delta)$, na\"ive Nystr\"om extension results in exact recovery with constant probability \cite{TR10}. Here $\mu$ is a measure of the coherence of the column space of $\mat{A}$ that differs slightly from the definition used in this paper. Their key observation is that if no columns of $\mat{A}$ are singularly influential, then $\mat{W}$ will have maximal rank when $\ell$ is slightly larger than the rank of $\mat{A}$. Thus, the number of samples required for exact recovery is determined by the rank of $\mat{A}$ and the coherence, $\mu,$ of its range space. They use a standard result from the compressed sensing literature to quantify this phenomenon and obtain an estimate for $\ell$ \cite{CR07}.

  In \cite{KLL10}, Kwok, Li, and Lu propose replacing $\mat{W}$ with a low-rank approximation $\tilde{\mat{W}}$ to facilitate the pseudoinversion operation. This large-scale variant of the na\"ive Nystr\"om extension allows a larger number of column samples to be drawn and leads to smaller empirical approximation errors. The approximation $\tilde{\mat{W}}$ is constructed using the randomized methodology espoused in \cite{HMT11}. The analysis of the error combines bounds provided in \cite{HMT11} with a matrix sparsification argument. In addition to $\ell$ and $k$, the Nystr\"om algorithm presented in \cite{KLL10} depends on two additional parameters that control the creation of $\tilde{\mat{W}}$: an oversampling factor $p$ and the number of iterations of the power method, respectively $p$ and $q.$ 
After taking $p = \ell -k$ and $q=1,$ the results of \cite{KLL10} provide error bounds for the na\"ive Nystr\"om extension:
\begin{align}
 \E\| \mat{A} - \mat{C}\mat{W}^\dagger \mat{C}^t\|_2 & \leq \mathrm{C} \left(\|\mat{A} - \mat{A}_k\|_2 + \frac{n}{\sqrt{l}} \max_i (\mat{A})_{ii} \right) \label{eqn:kll10spectralbound}\\
 \E\| \mat{A} - \mat{C}\mat{W}^\dagger \mat{C}^t\|_F & \leq \mathrm{C} \sqrt{\ell} \left( \|\mat{A} - \mat{A}_k\|_F + \frac{n}{\sqrt{\ell}} \max_i (\mat{A})_{ii} \right). \notag
\end{align}
 Of the works mentioned, only \cite{KLL10} provides a bound on the spectral error of the Nystr\"om method for $\mat{A}$ of arbitrary rank. Unfortunately, the quantity $\max_i (\mat{A})_{ii}$ is, for a general $\mat{A} \succeq 0,$ bounded only by $\lambda_1(\mat{A}).$ Thus equation \eqref{eqn:kll10spectralbound} does not provide a relative-error bound. In fact, the spectral norm error bound provided in this paper is always tighter than the bound provided in \cite{KLL10}, for any choice of $p$ and $q.$ 
 
Our work presents an intuitive and simple approach to the analysis of Nystr\"om extensions through their connection to the randomized column subset selection problem. This allows us to obtain the first truly relative-error guarantee on the spectral norm error. This paper analyzes the na\"ive sampling scheme but we believe that the framework given is flexible enough to be fruitfully applied to the analysis of other Nystr\"om extension schemes including, in particular, the large-scale variant introduced in \cite{KLL10}.

\subsection{Outline} In Section \ref{sec:notation} we introduce our notation and review some algebraic preliminaries. In Section \ref{sec:colselection} we establish a connection between the Nystr\"om extension procedure and the column subset selection problem. We exploit this connection and a result from \cite{HMT11} to provide a general error bound for any Nystr\"om extension scheme. In Section \ref{sec:naiveproof} we specialize this result to the case of the na\"ive Nystr\"om extension.

\section{Notation}
\label{sec:notation}

We work exclusively with real matrices and order the eigenvalues of a PSD matrix $\mat{A}$ so that $\lambda_1(\mat{A}) \geq \lambda_2(\mat{A}) \geq \cdots \geq \lambda_n(\mat{A}).$ Each PSD matrix $\mat{A}$ has a unique square root $\mat{A}^{1/2}$ that is also positive-semidefinite, has the same eigenspaces as $\mat{A},$ and satisfies $\mat{A} = \big(\mat{A}^{1/2}\big)^2.$

The projection onto the column space of a matrix $\mat{M}$ is written $\mat{P}_{\mat{M}}$ and satisfies
\[
\mat{P}_{\mat{M}} = \mat{M}\mat{M}^\dagger = \mat{M} (\mat{M}^t \mat{M})^\dagger \mat{M}^t.
\]
The notation $(\vec{x})_j$ refers to the $j$th entry of the vector $\vec{x},$ and $(\mat{M})_i$ refers to the $i$th column of the matrix $\mat{M}.$ Likewise, $(\mat{M})_{ij}$ refers to the $(i,j)$ entry of $\mat{M}.$

The \emph{coherence} of a matrix $\mat{U} \in \R^{n\times k}$ with orthonormal columns is, up to a scaling factor, the maximum of the squared Euclidean norms of its rows:
\[
 \mu_0(\mat{U}) \colonequals \frac{n}{k} \max\nolimits_i \|(\mat{U}^t)_i\|^2 = \frac{n}{k} \max\nolimits_i (\mat{U}\mat{U}^t)_{ii} = \frac{n}{k} \max\nolimits_i (\mat{P}_{\mat{U}})_{ii}.
\]
From the last equality, we see that the coherence is in fact an intrinsic property of the subspace spanned by $\mat{U}.$ Thus we refer to the coherence of a subspace without first choosing a particular orthogonal basis $\mat{U}.$

\section{The connection to the column subset selection problem}
\label{sec:colselection}

In this section we establish a fruitful connection between the performance of the Nystr\"om extension and the performance of randomized \emph{column subset selection}.

Given a matrix $\mat{M}$, the goal of column selection is to choose a small but informative subset $\mat{C}$ of the columns of $\mat{M}$ so that, after approximating $\mat{M}$ with the matrix obtained by projecting $\mat{M}$ onto the span of $\mat{C},$ the residual $(\mat{I} - \mat{P}_{\mat{C}})\mat{M}$ is small in some norm. In randomized column subset selection, the columns $\mat{C}$ are choosen randomly, either uniformly or according to some data-dependent distribution. Column subset selection has important applications in statistical data analysis and has been investigated by both the numerical linear algebra and the theoretical computer science communities. For an introduction to the column subset selection literature, biased towards approaches involving randomization, we refer the interested reader to the surveys \cite{MM10,MM11}.

Our first theorem establishes that the Nystr\"om extension of $\mat{A}$ is intimately related to the randomized column subset selection problem for $\mat{A}^{1/2}.$
We model the column sampling operation as follows: let $\mat{S}$ be a random matrix with $\ell$ columns, each of which has exactly one nonzero element. Then right multiplication by $\mat{S}$ selects  $\ell$ columns from $\mat{A}$:
\begin{equation*}
  \mat{C} = \mat{A} \mat{S} \quad \text{ and } \quad \mat{W} = \mat{S}^t \mat{A} \mat{S}.
\end{equation*}
The distribution of $\mat{S}$ reflects the type of sampling being performed. In the case of the na\"ive Nystr\"om extension,  $\mat{S}$ is distributed as the first $\ell$ columns of a uniformly random permutation matrix. 

We use the following partitioning of the eigenvalue decomposition of $\mat{A}$ to state our results:
\begin{equation}
\mat{A} = \bordermatrix[{[}{]}]{%
&^k \vspace{-0.75ex} & \!\!^{n-k}  \hspace{1ex}\\
& \vspace{0.25ex} \mat{U}_1 \hspace{-2ex} & \mat{U}_2 
}
\bordermatrix[{[}{]}]{%
& \vspace{-0.75ex} ^k &\!\!^{n-k} \hspace{1ex}\\
& \mat{\Sigma}_1 & \\
& \vspace{0.5ex} & \mat{\Sigma}_2 
}
\bordermatrix*[{[}{]}]{%
\mat{U}_1^t \!\! & \\
\vspace{-1ex} \mat{U}_2^t \!\! & \\
 \vspace{-1ex} &
}
\label{eqn:svdpartition}
\end{equation}

The columns of $\mat{U}_1$ and $\mat{U}_2$ respectively span a dominant $k$-dimensional invariant subspace of $\mat{A}$ and the corresponding bottom $(n-k)$-dimensional invariant subspace of $\mat{A}.$ The interaction of the column sampling matrix $\mat{S}$ with the invariant subspaces spanned by $\mat{U}_1$ and $\mat{U}_2$ is captured by the matrices
\begin{equation}
 \mat{\Omega}_1 = \mat{U}_1^t \mat{S}, \quad \mat{\Omega}_2 = \mat{U}_2^t \mat{S}.
\label{eqn:sampleprojections}
\end{equation}

\begin{thm}
Let $\mat{A}$ be a PSD matrix of size $n$ and let $\mat{S}$ be an $n \times \ell$ matrix. Partition $\mat{A}$ as in equation \eqref{eqn:svdpartition} and define $\mat{\Omega}_1$ and $\mat{\Omega}_2$ as in equation \eqref{eqn:sampleprojections}.

Assume $\mat{\Omega}_1$ has full row rank. Then the spectral approximation error of the Nystr\"om extension of $\mat{A}$ using $\mat{S}$ as the column sampling matrix satisfies
\begin{equation}
\|\mat{A} - \mat{C}\mat{W}^\dagger \mat{C}^t\|_2 = \|\big(\mat{I} - \mat{P}_{\mat{A}^{1/2} \mat{S}}\big)\mat{A}^{1/2}\|_2^2 \leq \norm{\mat{\Sigma}_2}_2 \left( 1 + \|\mat{\Omega}_2 \mat{\Omega}_1^\dagger\|_2^2\right).
  \label{eqn:spectralnystromerrbound}
\end{equation}
\label{thm:colselection}
\end{thm}
Prior analyses of Nystr\"om extensions have used the Cholesky decomposition of $\mat{A}.$ By instead using the square-root, Theorem \ref{thm:colselection} establishes an equivalence between the column subset selection problem and the Nystr\"om extension procedure and gives a deterministic relative-error bound on the performance of Nystr\"om extensions. 

To establish Theorem \ref{thm:colselection}, we use the following bound on the error incurred by projecting a matrix onto a random subspace of its range (\cite[Theorem 9.1]{HMT11}).

\begin{prop}
 Let $\mat{M}$ be a PSD matrix of size $n$. Fix integers $k$ and $\ell$ satisfying $1 \leq k \leq l \leq n.$ 

Let $\mat{U}_1$ and $\mat{U}_2$ be matrices with orthogonal columns spanning, respectively, a dominant $k$-dimensional invariant subspace of $\mat{M}$ and the corresponding bottom $(n-k)$-dimensional invariant subspace of $\mat{M}.$ Let $\mat{\Sigma}_2$ be the diagonal matrix of eigenvalues corresponding to the bottom $(n-k)$-dimensional invariant subspace of $\mat{M}.$ 

Given a matrix $\mat{S}$ of size $n \times \ell$, define $\mat{\Omega_1} = \mat{U}_1^t \mat{S}$ and $\mat{\Omega_2} = \mat{U}_2^t \mat{S}.$ Then, assuming that $\mat{\Omega}_1$ has full row rank,
\[
\| (\mat{I} - \mat{P}_{\mat{M}\mat{S}}) \mat{M} \|_2^2 \leq \| \mat{\Sigma}_2 \|_2^2 + \big\| \mat{\Sigma}_2 \mat{\Omega}_2 \mat{\Omega}_1^\dagger \big \|_2^2.
\]
 \label{prop:deterministicerrorbound}
\end{prop}

\begin{proof}[Proof of Theorem \ref{thm:colselection}]
We write the Nystr\"om extension in terms of the square root of $\mat{A}$ and a projection onto the space spanned by $\mat{A}^{1/2} \mat{S}:$
\begin{align*}
  \mat{C} \mat{W}^\dagger \mat{C}^t & = \mat{A} \mat{S} (\mat{S}^t \mat{A} \mat{S})^\dagger \mat{S}^t \mat{A} \\
  & = \mat{A}^{1/2} [\mat{A}^{1/2} \mat{S} (\mat{S}^t \mat{A}^{1/2} \mat{A}^{1/2} \mat{S})^\dagger \mat{S}^t \mat{A}^{1/2}] \mat{A}^{1/2} \\
  & = \mat{A}^{1/2} \mat{P}_{\mat{A}^{1/2} \mat{S}} \mat{A}^{1/2}.
\end{align*}
It follows that the spectral error of the Nystr\"om extension satisfies
\begin{align*}
  \norm{\mat{A} - \mat{C} \mat{W}^\dagger \mat{C}^t}_2 & = \norm{\mat{A}^{1/2} ( \mat{I} - \mat{P}_{\mat{A}^{1/2}\mat{S}}) \mat{A}^{1/2}}_2  = \norm{\mat{A}^{1/2} ( \mat{I} - \mat{P}_{\mat{A}^{1/2}\mat{S}})^2 \mat{A}^{1/2}}_2\\ 
 & = \norm{( \mat{I} - \mat{P}_{\mat{A}^{1/2}\mat{S}}) \mat{A}^{1/2}}_2^2.
\end{align*}
The second equality holds because of the idempotency of projections. The third follows from the fact that $\|\mat{A}\mat{A}^t\|_2 = \|\mat{A}\|_2^2$ for any matrix $\mat{A}.$ Partition $\mat{A}$ as in equation \eqref{eqn:svdpartition}. Equation \eqref{eqn:spectralnystromerrbound} now follows immediately from Proposition \ref{prop:deterministicerrorbound} with $\mat{M} = \mat{A}^{1/2}.$
\end{proof}

\section{ Error bounds for na\"ive Nystr\"om extension}
\label{sec:naiveproof}

In this section, we provide a bound on the spectral norm approximation error of na\"ive Nystr\"om extensions. For convenience, we recall the following partitioning of the eigenvalue decomposition of a PSD matrix of size $n$:
\begin{equation}
 \mat{A} = \bordermatrix[{[}{]}]{%
&^k \vspace{-0.75ex} & \!\!^{n-k}  \hspace{1ex}\\
& \vspace{0.25ex} \mat{U}_1 \hspace{-2ex} & \mat{U}_2 
}
\bordermatrix[{[}{]}]{%
& \vspace{-0.75ex} ^k &\!\!^{n-k} \hspace{1ex}\\
& \mat{\Sigma}_1 & \\
& \vspace{0.5ex} & \mat{\Sigma}_2 
}
\bordermatrix*[{[}{]}]{%
\mat{U}_1^t \!\! & \\
\vspace{-1ex} \mat{U}_2^t \!\! & \\
 \vspace{-1ex} &
},
\label{eqn:svdpartition2}
\end{equation}
where the columns of $\mat{U}_1$ and $\mat{U}_2$ respectively span a dominant $k$-dimensional invariant subspace of $\mat{A}$ and the corresponding bottom $(n-k)$-dimensional invariant subspace of $\mat{A}.$ We also recall the matrices 
\begin{equation}
 \mat{\Omega}_1 = \mat{U}_1^t \mat{S}, \quad \mat{\Omega}_2 = \mat{U}_2^t \mat{S}
\label{eqn:sampleprojections2}
\end{equation}
that capture the interaction of the column sampling operation with the invariant subspaces of $\mat{A}.$
Theorem \ref{thm:uniformnystromerror} establishes that, if the spectrum of $\mat{A}$ decays sufficiently and an appropriate number of columns are sampled, then the error incurred by the na\"ive Nystr\"om extension process is small.

\begin{thm}
 Let $\mat{A}$ be a PSD matrix of size $n.$ Given an integer $k \leq n,$ partition $\mat{A}$ as in equation \eqref{eqn:svdpartition2}. Let
$\tau$ denote the coherence of $\mat{U}_1,$
\[
 \tau = \mu_0(\mat{U}_1) .
\]

Fix a failure probability $\delta \in (0,1).$ For any $\varepsilon \in (0,1),$ if
\[
 \ell \geq \frac{2 \tau k \log\left(\frac{k}{\delta}\right)}{(1-\varepsilon)^2}
\]
columns of $\mat{A}$ are chosen uniformly at random and used to form a Nystr\"om extension, the spectral norm error of the approximation satisfies
\[
\|\mat{A} - \mat{C} \mat{W}^\dagger \mat{C}^t\|_2 \leq \lambda_{k+1}(\mat{A}) \left(1 + \frac{ n}{\varepsilon \ell} \right) 
\]
with probability exceeding $1-\delta.$
\label{thm:uniformnystromerror}
\end{thm}

\begin{remark}
The coherence of $\mat{U}_1$ is a measure of how much comparative influence the individual columns of $\mat{A}$ have over the dominant $k$-dimensional invariant subspace of $\mat{A}$ spanned by $\mat{U}_1$: if $\tau$ is small, then all columns have essentially the same influence; if $\tau$ is large, then it is possible that there is a single column in $\mat{A}$ which alone determines one of the top $k$ eigenvectors of $\mat{A}.$

For illustrative purposes, we point out that the coherence of a random $k \times n$ orthogonal matrix, i.e. a matrix distributed uniformly on the Stiefel manifold, is $\mathrm{O}(\max(k, \log n)/k)$ with high probability \cite{CR09}. The coherence of an arbitrary $\mat{U}_1$ is no smaller than $1,$ and may be as large as $\frac{n}{k}.$ 
\end{remark}

\begin{remark}
 Theorem \ref{thm:uniformnystromerror}, like the main result of \cite{TR10}, promises exact recovery when $\mat{A}$ is exactly rank $k$ and has small coherence, with a sample of $\mathrm{O}(k \log k)$ columns. Unlike the result in \cite{TR10}, Theorem \ref{thm:uniformnystromerror} is applicable in the case that $\mat{A}$ is full-rank but has a sufficiently fastly decaying spectrum.
\end{remark}



\begin{remark}
We might ask where attempts at sharpening the analysis of the na\"ive Nystr\"om extension should be aimed: toward more refined linear algebra bounds on column selection (Proposition~\ref{prop:deterministicerrorbound}), or toward a deeper analysis of the randomness (Theorem \ref{thm:uniformnystromerror})?

 It is known that, for uniform sampling, the quantity $\norm{\mat{\Omega}_1^\dagger}_2^2$ remains $\Omega(n/\ell)$ once $\ell \gtrsim \tau k \log k$ \cite{R99}. Likewise, in the regime of low $\tau$ and $k \ll n$, $\mat{\Omega}_2$ is likely to be an almost isometric embedding. This follows from the fact that the columns of $\mat{\Omega}_2$ contain $n-k$ of the $n$ entries of the corresponding columns of $\mat{U}_2$, so they are likely to be almost orthogonal and linearly independent.  Together, these observations suggest that $\norm{\mat{\Omega}_2 \mat{\Omega}_1^\dagger}_2$ remains $\Omega(n/\ell)$ also. Thus, we expect that no bounds much sharper than Theorem \ref{thm:uniformnystromerror} on the error of the na\"ive Nystr\"om extension can be derived using the algebraic results in Proposition~\ref{prop:deterministicerrorbound} as the starting point. 
\end{remark}

To obtain Theorem \ref{thm:uniformnystromerror}, we use Theorem \ref{thm:colselection} in conjunction with a bound on $\norm{\mat{\Omega}_1^\dagger}_2^2$ provided by the following lemma. 

\begin{lemma}
Let $\mat{U}$ be an $n \times k$ matrix with orthonormal columns. Take $\tau$ to be the coherence of $\mat{U}, $
\[
 \tau = \mu_0(\mat{U}).
\]
Select $\varepsilon \in (0,1)$ and a nonzero failure probability $\delta.$ Let $\mat{S}$ be a random matrix distributed as the first $\ell$ columns of a uniformly random permutation matrix of size $n,$ where
\[
 \ell \geq \frac{2\tau}{(1-\varepsilon)^2}k\log\frac{k}{\delta}.
\]
Then with probability exceeding $1- \delta$, the matrix $\mat{U}^t\mat{S}$ has full row rank and satisfies
\[
 \norm{(\mat{U}^t \mat{S})^\dagger}_2^2 \leq \frac{n}{\varepsilon l}.
\]
 \label{lem:omega1normbound}
\end{lemma}

We now proceed with the proof of Theorem \ref{thm:uniformnystromerror}.

\begin{proof}[Proof of Theorem \ref{thm:uniformnystromerror}]
Because we are using uniform sampling without replacement, the sampling matrix $\mat{S}$ is formed by taking the first $\ell$ columns of a uniformly sampled random permutation matrix. Note that by Lemma \ref{lem:omega1normbound}, $\mat{\Omega}_1$ has full row rank, so the bounds in Theorem \ref{thm:colselection} are applicable.

Applying Lemma \ref{lem:omega1normbound}, we see that $\norm{\mat{\Omega}_1^\dagger}_2^2 \leq \frac{n}{\varepsilon \ell}$ with probability exceeding $1-\delta.$ From Theorem \ref{thm:colselection}, we conclude that
 \[
  \|\mat{A} - \mat{C} \mat{W}^\dagger \mat{C}^t\|_2 \leq \norm{\mat{\Sigma}_2}_2 \left(1 + \norm{\mat{\Omega}_2}_2^2 \norm{\mat{\Omega}_1^\dagger}_2^2 \right) \leq \lambda_{k+1}(\mat{A})\left(1 + \frac{n}{\varepsilon \ell}\right)
 \]
with at least the same probability. To obtain the second inequality, we used the fact that $\norm{\mat{\Omega}_2}_2 \leq  \norm{\mat{U}_2}_2 \norm{\mat{\Omega}}_2 \leq 1.$

\end{proof}

One potential source of difficulty in the proof of Lemma \ref{lem:omega1normbound} is the fact that the columns are sampled without replacement, which introduces dependencies among the entries of the sampling matrix $\mat{S}.$ The following matrix Chernoff bound, a standard simplification of the lower Chernoff bound developed in \cite[Theorem 2.2]{T11}, allows us to gloss over these dependencies.

\begin{prop}
 Let $\mathcal{X}$ be a finite set of PSD matrices with dimension $k,$ and suppose that 
\[
 \max_{\mat{X} \in \mathcal{X}} \lambda_1(\mat{X}) \leq B. 
\]
 Sample $\{\mat{X}_1, \ldots, \mat{X}_\ell\}$ uniformly at random from $\mathcal{X}$ without replacement. Compute
\[
 \mu_{\text{min}} = \ell \cdot \lambda_k(\E \mat{X}_1).
\]

Then
\[
 \Prob{ \lambda_k\left(\sum_i \mat{X}_i \right) \leq \varepsilon\mu_{\text{min}} } \leq k \cdot \e^{-(1-\varepsilon)^2 \mu_{\text{min}}/(2B)} \quad \text{for } \varepsilon \in [0,1].
\]

 \label{prop:chernoffworeplacement}
\end{prop}
 
\begin{proof}[Proof of Lemma \ref{lem:omega1normbound}]
 Note that $\mat{U}^t \mat{S}$ has full row rank if $\lambda_k(\mat{U}^t \mat{S}\mat{S}^t \mat{U}) > 0.$ Furthermore,
\[
 \norm{(\mat{U}^t \mat{S})^\dagger}_2^2 = \lambda_k^{-1}(\mat{U}^t \mat{S} \mat{S}^t \mat{U}).
\]
Thus to obtain both conclusions of the lemma, it is sufficient to verify that 
\[
\lambda_k(\mat{U}^t \mat{S}\mat{S}^t \mat{U}) \geq \frac{\varepsilon \ell}{n}
\]
when $\ell$ is as stated.

We apply Proposition \ref{prop:chernoffworeplacement} to bound the probability that this inequality is not satisfied. Let $\vec{u}_i$ denote the $i$th column of $\mat{U}^t.$ Then \[
 \lambda_k(\mat{U}^t \mat{S}\mat{S}^t \mat{U}) = \lambda_k\left( \sum_{i=1}^\ell \mat{X}_i \right),
\]
where the $\mat{X}_i$ are chosen uniformly at random, without replacement, from the set $\mathcal{X} = \{ \vec{u}_i \vec{u}_i^t \}_{i=1,\ldots,n}.$ Clearly
\[
 B = \max_i \|\vec{u}_i\|^2 = \frac{k}{n}\tau \quad \text{ and } \quad  \mu_{\text{min}} = \ell \cdot\lambda_k(\E \mat{X}_1) = \frac{\ell}{n} \lambda_k(\mat{U}^t\mat{U}) = \frac{\ell}{n}.
\]
Proposition \ref{prop:chernoffworeplacement} yields
\[
 \Prob{\lambda_k\left(\mat{U}^t \mat{S}\mat{S}^t \mat{U} \right) \leq \varepsilon \frac{\ell}{n}} \leq k \cdot \e^{-(1-\varepsilon)^2 \ell/(2 k \tau)}.
\]

We require enough samples that 
\[
\lambda_k(\mat{U}^t \mat{S}\mat{S}^t \mat{U}) \geq \varepsilon \frac{\ell}{n}
\]
with probability greater than $1 - \delta$, so we set 
\[
 k \cdot \e^{-(1-\varepsilon)^2 \ell/(2 k \tau)} \leq \delta 
\]
and solve for $\ell,$ finding
\[
 \ell \geq \frac{2\tau}{(1-\varepsilon)^2} k\log \frac{k}{\delta}.
\]

Thus, for values of $\ell$ satisfying this inequality, we achieve the stated spectral error bound and ensure that $\mat{U}^t\mat{S}$ has full row rank.
\end{proof}

\bibliographystyle{amsalpha}
\bibliography{nystrom-method-a-new-analysis}
\end{document}